\documentclass[11pt]{amsart}
\usepackage{amssymb,amsthm,amsmath,amstext}
\usepackage{mathrsfs}  
\usepackage{bm}        
\usepackage{mathtools} 
\usepackage{color}
\usepackage[bbgreekl]{mathbbol}

\usepackage{cite}

\usepackage[left=1.5in,top=1in,right=1.5in,bottom=1in]{geometry}

\usepackage{graphicx}

\usepackage[all]{xy}
\usepackage{tikz}

\newcommand{\xycenter}[1]{
	\begin{center}
	\mbox{\xymatrix{#1}}
	\end{center}
	}

\theoremstyle{plain}
\newtheorem{theorem}{Theorem}[section]

\newtheorem{proposition}[theorem]{Proposition}
\newtheorem{lemma}[theorem]{Lemma}

\newtheorem{problem}[theorem]{Problem}

\theoremstyle{definition}
\newtheorem{definition}[theorem]{Definition}
\newtheorem{notation}[theorem]{Notation}

\newtheorem*{maintheorem}{Theorem \ref{tSaturatedPrelogY}}

\theoremstyle{remark}
\newtheorem{remark}[theorem]{Remark}

\newcommand{\ioUpper}[2]{\iota^*_{\{#1\}>\{#2\}}}
\newcommand{\ioLower}[2]{\iota_{\{#1\}>\{#2\}*}}

\newcommand{\sheaf}[1]{\mathscr{#1}}

\newcommand{\LL}{\sheaf{L}}
\newcommand{\OO}{\sheaf{O}}
\newcommand{\QQ}{\sheaf{Q}}
\newcommand{\MM}{\sheaf{M}}
\newcommand{\EE}{\sheaf{E}}

\newcommand{\NN}{\sheaf{N}}
\newcommand{\VV}{\sheaf{V}}
\newcommand{\ZZ}{\sheaf{Z}}

\newcommand{\YY}{\sheaf{Y}}
\newcommand{\XX}{\sheaf{X}}

\newcommand{\TT}{\sheaf{T}}

\newcommand{\Z}{\mathbb Z}
\newcommand{\N}{\mathbb N}

\newcommand{\C}{\mathbb C}
\renewcommand{\P}{\mathbb P}
\newcommand{\Q}{\mathbb Q}

\DeclareMathOperator{\coker}{\mathrm{coker}}
\DeclareMathOperator{\Chow}{\mathrm{CH}}
\DeclareMathOperator{\Num}{\mathrm{Num}}
\DeclareMathOperator{\Chownum}{\mathrm{Num}}

\DeclareMathOperator{\Chowprelognum}{\mathrm{Num}^*_{\mathrm{prelog}}}

\DeclareMathOperator{\Numsatprelog}{\mathrm{Num}_{\mathrm{prelog, sat}}}
\DeclareMathOperator{\weight}{wt}

\newcommand{\CH}{\mathrm{CH}}

\usepackage[backref=page]{hyperref}
\hypersetup{pdftitle={Prelog Chow groups of self-products of degenerations of cubic threefolds}}
\hypersetup{pdfauthor={Christian Boehning, Hans-Christian Graf von Bothmer, Michel van Garrel}}
\hypersetup{colorlinks=true,linkcolor=blue,anchorcolor=blue,citecolor=blue,letterpaper=true}

\newcommand*{\DashedArrow}[1][]{\mathbin{\tikz [baseline=-0.25ex,-latex, dashed,#1] \draw [#1] (0pt,0.5ex) -- (2.3em,0.5ex);}}

\begin{document}


\title[Prelog Chow groups of self-products of cubic threefolds]{Prelog Chow groups of self-products of degenerations of cubic threefolds}



\author[B\"ohning]{Christian B\"ohning}
\address{Christian B\"ohning, Mathematics Institute, University of Warwick\\
Coventry CV4 7AL, United Kingdom}
\email{C.Boehning@warwick.ac.uk}

\author[von Bothmer]{Hans-Christian Graf von Bothmer}
\address{Hans-Christian Graf von Bothmer, Fachbereich Mathematik der Universit\"at Hamburg\\
Bundesstra\ss e 55\\
20146 Hamburg, Germany}
\email{hans.christian.v.bothmer@uni-hamburg.de}

\author[van Garrel]{Michel van Garrel}
\address{Michel van Garrel, School of Mathematics, University of Birmingham\\
Edgbaston,
Birmingham, B15 2TT,
United Kingdom}
\email{m.vangarrel@bham.ac.uk}

\date{\today}


\begin{abstract}
It is unknown whether smooth cubic threefolds have an (integral Chow-theoretic) decomposition of the diagonal, or whether they are stably rational or not in general. 
As a first step towards making progress on these questions, we compute the (saturated numerical) prelog Chow group of the self-product of a certain degeneration of cubic threefolds. 
\end{abstract}

\maketitle

\section{Introduction}\label{sIntroduction}
A large area within the study of the birational geometry of rationally connected varieties is concerned with varieties that are \emph{close} to projective space. This can for example be made precise as follows. 

\begin{definition}\label{def:rat}
Let $V$ be a variety over $\C$. $V$ is said to be 
\begin{itemize}
\item \emph{rational} if it is birational to a projective space $\P^{n}$,
\item \emph{stably rational} if $V\times\P^{m}$ is rational for some $m\geq0$,
\item \emph{unirational} if there is a dominant rational morphism $\P^{n}\DashedArrow V$ for some $n\geq0$.
\end{itemize}
\end{definition}

Note that $(\rm{rational}) \Rightarrow (\rm{stably \, \, rational}) \Rightarrow (\rm{unirational})$.
The only known method to separate rational varieties from strictly stably rational ones is the one given in  \cite{B-CT-S-S85}.  
There is however a powerful method principally developed by Voisin \cite{Voi15} and Colliot-Th\'{e}l\`{e}ne/Pirutka \cite{CT-P16} that can be employed to show that a rationally connected variety $V$ is not stably rational (i.e.\ is stably irrational). This makes use of the following notion.

\begin{definition}
Let $V$ be a smooth $n$-dimensional rationally connected projective variety over $\C$. 
We say that $V$ admits a \emph{Chow-theoretic decomposition of the diagonal} if one can write
\begin{equation}\label{eq:def:decomp}
[\Delta_V] = [V \times {\rm pt}] + [Z] \text{ in the Chow group } \CH_{n}(V\times V),
\end{equation}
where $\Delta_V$ is the diagonal, ${\rm pt}$ is a point of $V$ and $Z$ is a cycle supported on $D\times V$ for some divisor $D$ in $V$. Less restrictively, $V$ is said to admit a \emph{cohomological decomposition of the diagonal} if \eqref{eq:def:decomp} holds in cohomology with $\Z$-coefficients. 
\end{definition}

It follows right away that projective $n$-space $\P^n$ (or more generally any rational homogeneous variety) has a (Chow-theoretic) decomposition of the diagonal since the diagonal is rationally equivalent to a sum of products of Schubert varieties with their dual Schubert varieties; moreover, the existence of a decomposition of the diagonal is a stable birational invariant of smooth projective varieties. Note that while powerful, this is also a rather subtle invariant. Indeed, if $V$ is a unirational variety with unirational parametrisation $\P^n\DashedArrow V$ of degree $N$, then there is a decomposition
\begin{equation}\label{eq:rat-decomp}
N \, [\Delta_V] = N \, [V \times {\rm pt}] + [Z] \text{ in } \CH_n(V\times V),
\end{equation}
with $Z$ supported on $D\times V$ for $D$ a divisor in $V$. In fact, if one allows the passage to rational coefficients, any rationally connected variety admits such a rational decomposition of the diagonal. 

Voisin's method to prove stable irrationality of $V$ now proceeds by showing that a resolution of singularities of a suitable degeneration of $V$ with some mild class of singularities does not admit a decomposition of the diagonal (the original version required the degeneration to be integral and nodal, but this has subsequently been relaxed in several ways). In fact, the circumstance that in Voisin's method one has to restrict the singularities of the degeneration in some way is the primary obstacle to broadening its range of applications. 

Voisin's original method has been generalised in various ways subsequently, in each case mainly with the goal in mind to allow wider (more singular) classes of degenerations; here we just mention the articles \cite{CT-P16, To16, Schrei19-1, NiOt19, NiShi19} and references therein. However, these methods are so far not capable of shedding any light on the following problem.

\begin{problem}\label{pCubics}
Is a very general cubic hypersurface $V$ of dimension $n\ge 3$ stably irrational? Is a very general such cubic of dimension $n\ge 4$ non-rational? Are there any smooth cubic hypersurfaces of odd dimension $\ge 5$ that are rational?
\end{problem}

The answers to all of these questions are completely unknown; the most common viewpoint is probably that the answers to the first two questions should be yes, and the answer to the last question no. It is known that all smooth cubic hypersurfaces of dimensions $\ge 2$ are unirational and that in even dimensions there are some rational examples (e.g. those containing two disjoint linear subspaces of half their dimension, whence one can use the third intersection point map to get a rational parametrisation). Moreover, Clemens--Griffiths in \cite{CG72} prove that any smooth cubic threefold is non-rational. Denote by $J(V)$ the intermediate Jacobian of $V$ and let $\theta\in{\rm H}^2(J(V),\Z)$ be the class of the theta divisor of $J(V)$. $(J(V),\theta)$ is a principally polarised abelian variety. Then, more precisely, Clemens--Griffiths show that $(J(V),\theta)$ is not a product of Jacobians of curves from which irrationality of $V$ follows in a relatively straightforward way using weak factorisation of birational maps into blowups in smooth centres and the inverses of such maps. 

\medskip

To make progress on the problem whether (all, or very general) smooth cubic threefolds $V$ are stably irrational, Voisin in \cite{Voi13,Voi17} investigated the existence of a decomposition of the diagonal on them. In particular, she showed that some of them do admit a decomposition of the diagonal. Also, in Theorem 1.1 in \cite{Voi17} Voisin proves that $V$ admits a Chow-theoretic decomposition of the diagonal if and only if it admits a cohomological decomposition of the diagonal. However, the cubic threefolds that are known to admit a decomposition of the diagonal form a proper closed subset of the moduli space (by which we mean the quotient of the parameter space of smooth cubic threefolds by the action of projectivities). More precisely: 

\begin{theorem}[Voisin, Theorem 1.7 and Theorem 4.5 in \cite{Voi17}]
A smooth cubic threefold $V$ admits a decomposition of the diagonal if and only if the class $\theta^4/4!$ on $J(V)$ is algebraic. On the moduli space of smooth cubic threefolds, this algebraicity is satisfied (at least) on a countable union of proper closed subvarieties of codimension $\leq 3$.
\end{theorem}

Notice that this result leaves open the possibility that a very general cubic threefold $V$ does not admit a decomposition of the diagonal since these are proper subvarieties. 

\

In \cite{BBvG19-1} we started to develop a conceptual framework for an extension of Voisin's method that seeks to allow more singular, even reducible degenerations. The basic idea is that if $\pi \colon \VV \to B$ is a degeneration over a base $B$ with distinguished point $0\in B$, which we assume to be strictly semistable, then the existence of a decomposition of the diagonal on a very general fibre of $\VV$ will imply the existence of a \emph{log decomposition of the diagonal} on the central fibre $\VV_0$ of the family. The existence of such a log decomposition is a more stringent condition than the existence of a mere decomposition of the diagonal, hence it gives a more potent obstruction. The underlying intuition is that the notion of a log decomposition of the diagonal for $\VV_0$ should include that certain obstructions for the decomposition to deform in the family are already absent- thus a log decomposition of the diagonal should be something that ``carries within itself the potentiality to deform in the family to a decomposition on smooth fibres outside $0$". To make this idea precise takes some work and requires to assemble quite a bit of vocabulary- we refer the reader to \cite{BBvG19-1} and the recollection in Section \ref{sRecollection} for that. Also we feel that eventually the right theoretical setting should be the language of log structures in the sense of Fontaine-Illusie, Deligne-Faltings, K. Kato et al., but we have not yet succeeded in casting our ideas in this framework: this is the reason why we prefer to talk about prelog Chow rings, prelog Chow groups etc., the added ``pre" suggesting that this theory could be a first approximation to a better more powerful point of view. 

In any event, the upshot is that a prelog decomposition of the diagonal gives, in particular, a relation between cycle classes in a group which we called the saturated prelog Chow group associated to a degeneration in \cite{BBvG19-1}; there is a numerical version of this, which contains a homomorphic image of this saturated prelog Chow group, and which is better suited for concrete computations. There is also a prelog Chow ring (and a numerical version of it) associated to a given degeneration. What is necessary for the purposes of the present paper is recalled in Section \ref{sRecollection}.

In this article, we now consider a degeneration of a very general cubic threefold into the union of a hyperplane and a quadric in $\P^4$. The resulting family is not strictly semistable but can be made so after blowing up some non-Cartier divisors in the total space (that are components of the central fibre) and after potentially shrinking the base of the family. Taking the self-product of this family and birationally modifying it to make it again strictly semistable, we end up with a family $\YY \to B$ to which our theory applies, the very general fibre of which is the self-product of a very general smooth cubic threefold, and whose central fibre is simple normal crossing with four components described more precisely in Section \ref{sProductFamily}; in particular, all of those components are rational, and their numerical Chow groups are straightforward if somewhat tedious to compute. This allows us to eventually compute the saturated numerical prelog Chow group $\Num_{\mathrm{prelog, sat}}^3 (Y)$ in the relevant degree $3$ of the central fibre of this family. Our main result is then

\begin{maintheorem} \label{thm:main}
$\Num_{\mathrm{prelog, sat}}^3 (Y)\simeq \Z^6$ and is generated by the classes in Theorem \ref{tGeneratorsY} and a half of their sum. 
\end{maintheorem}

This is clearly just a very first step towards deciding the existence or non-existence of prelog decompositions of the diagonal on the central fibre of the degeneration under consideration. We like to see the main accomplishment of this article in the fact that one can actually do computations in a highly nontrivial case with the theoretical framework and notions introduced in \cite{BBvG19-1}, and secondly, that the numerical prelog Chow ring and saturated numerical prelog Chow group of the central fibre, however naive their definition may seem, do carry a great deal of information about and are a rather accurate reflection of the Chow group of the very general fibre of the family, compare Remark \ref{rVoisin}, Theorem \ref{tGeneratorsY}, Proposition \ref{pSpecialisationCyclesOnVxV} and Theorem \ref{tSaturatedPrelogY}. 

\

We write down explicit generators for both $\Num_{\mathrm{prelog}}^3(Y)/({\rm torsion})$ as well as $\Num_{\mathrm{prelog, sat}}^3 (Y)$ in Theorems \ref{tGeneratorsY} and \ref{tSaturatedPrelogY}. Checking these are indeed generators relies on computations \cite{BBG-M2} that we perform using the algebraic geometry software system Macaulay2 (M2). Additional calculations using M2 involve finding an explicit description for a certain numerical Chow ring in Proposition \ref{pGeneratorsY2}, solving linear equations in Lemma \ref{lPullbackD} and computing the rank of a certain matrix in Theorem \ref{tCubicThreefoldDegeneration}. We describe how these computations work in the course of the proofs of the corresponding results. Our code \cite{BBG-M2} contains extensive comments describing the computations. Instructions on how to install and use M2 are found on \url{www.macaulay2.com}. There is a web interface \url{www.unimelb-macaulay2.cloud.edu.au} where it can be tried out. In addition to \url{https://faculty.math.illinois.edu/Macaulay2/GettingStarted}, an extensive introduction to M2 is provided in the textbook \cite[Part 1]{E-G-S-S01}.

\section*{Acknowledgement}

We thank Lawrence Barrott, Jean-Louis Colliot-Th\'el\`ene, Alessio Corti, Mathieu Florence, Mark Gross, Helge Ruddat, Stefan Schreieder, Bernd Siebert and Claire Voisin for valuable discussions surrounding the ideas presented in this paper. This project has received funding from the European Union's Horizon 2020 research and innovation programme under the Marie Sklodowska-Curie grant agreement No 746554 and has been supported by Dr.\ Max R\"ossler, the Walter Haefner Foundation and the ETH Z\"urich Foundation. The first author would like to acknowledge the stimulating environment and discussions of the 2019 AIM Workshop Rationality Problems in Algebraic Geometry that were helpful for coming to terms with various aspects of the present paper.

\section{Recollection of prelog Chow rings and saturated prelog Chow groups}\label{sRecollection}

We work over the complex numbers $\C$ throughout. In this Section we give a summary of those concepts developed in \cite{BBvG19-1} that are used in this paper.

Let $X$ be a simple normal crossing variety with irreducible components $X_i$, $i\in I$, and normalisation $\nu \colon X^{\nu}\to X$. For a subset $A\subset I$ we put
\[
X_A := \bigcap_{i\in A} X_i. 
\]
To simplify notation, we write 
\[
X_{i_1 i_2\dots i_k} := X_{\{i_1, i_2, \dots , i_k\} } .
\]
In particular, the $X_{ij}=X_i\cap X_j$ are smooth since $X$ is simple normal crossing. We define the following ring:

\small
\begin{gather*}
R (X) := R_{\mathrm{num}} (X) := \Bigl\{ (\alpha_i) \in \bigoplus_i \mathrm{Num}^* (X_i)\, \Bigr| \Bigl. \,\forall\, i,j: \,\alpha_i\mid_{X_{ij}} = \alpha_j\mid_{X_{ij}} \: \mathrm{in}\: \mathrm{Num}^* (X_{ij}) \Bigr\} .
\end{gather*}
\normalsize
Here we simply write a restriction symbol to denote pull-backs to $X_{ij}$. We call the condition $\alpha_i\mid_{X_{ij}} = \alpha_j\mid_{X_{ij}}$ the \emph{prelog condition}, and the above ring the \emph{ring of compatible classes}. We will use the simpler notation $R (X)$ instead of $R_{\mathrm{num}} (X)$ below since numerical equivalence is all we will consider here. 

\medskip

The following condition holds for an snc variety $X$ that is smoothable with smooth total space.

\begin{definition}\label{dFriedman}
Let $X$ be a snc variety with at worst triple intersections. 
We say that $X$ satisfies the \emph{Friedman condition} if for every intersection $X_{ij}$ we have
\[
\NN_{X_{ij}/X_i}\otimes \NN_{X_{ij}/X_j} \otimes \OO (T) = \OO_{X_{ij}}.
\]
Here $T$ is the union of all triple intersections $X_{ijk}$ that are contained in $X_{ij}$.
\end{definition}

We will use this later on as a sanity for some of our computations.

\medskip

We denote by $\iota_{ A > B } \colon X_A \to X_B$ for $B \subset A \subset I$ the inclusion map. When enumerating the elements of the sets $A, B$ below we omit separating commas to simplify notation when there is no risk of confusion. Thus for example we write $\iota_{\{ij\}, \{j\}}$ when $A=\{i, j\}$, $B=\{j\}$.

\begin{definition}\label{dBothmer}
Let $X$ be an snc variety that has at worst triple intersections and satisfies the Friedman condition. Then we define $ \Chowprelognum (X)$ via the following  diagram:
 \xycenter{
  & 0 \ar[d] & 0 \ar[d] & \\
  & R(X) \ar[d] \ar[r] & \Chowprelognum (X) \ar[d] \ar[r] & 0 \\
\bigoplus \Chownum^*(X_{ij}) \ar[d]^-{\rho'} \ar[r]^\delta
&\bigoplus \Chownum^*(X_{i})  \ar[r]^-{} \ar[d]^-\rho
&\coker (\delta ) \ar[r]
&  0 \\
\bigoplus \Chownum^*(X_{ijk}) \ar[r]^-{\delta'}
&\bigoplus \Chownum^*(X_{ij})
 }
 
\newcommand{\II}{\{1,\dots,n\}}

Here the maps $\rho , \rho',\delta , \delta'$ are defined as follows, using the convention $a<b<c$, $i<j<k$:

\begin{align*}
\bigl(\delta(z_{ij})\bigr)_a 	
	&= \left\{ \begin{matrix}  
	     \ioLower{ij}{i}(z_{ij}) & \text{if $a=i$,} \\
	     -\ioLower{ij}{j}(z_{ij}) & \text{if $a=j$,} \\
	     0 & \text{otherwise.}
	\end{matrix} \right.\\
	\bigl(\rho(z_{i})\bigr)_{ab} 	
	&= \left\{ \begin{matrix}  
	     \ioUpper{ab}{i}(z_{i}) & \text{if $i=a$} \\
	     -\ioUpper{ab}{i}(z_{i}) & \text{if $i=b$} \\
	     0 & \text{otherwise}
	\end{matrix} \right. \\
	\bigl(\rho'(z_{ij})\bigr)_{abc} 	
	&= \left\{ \begin{matrix}  
	     \ioUpper{abc}{ij}(z_{ij}) & \text{if $(i,j)=(a,b)$} \\
	     -\ioUpper{abc}{ij}(z_{ij}) & \text{if $(i,j)=(a,c)$} \\
	     \ioUpper{abc}{ij}(z_{ij}) & \text{if $(i,j)=(b,c)$} \\
	     0 & \text{otherwise}
	\end{matrix} \right. \\
	\bigl(\delta'(z_{ijk})\bigr)_{ab} 	
	&= \left\{ \begin{matrix}  
	     -\ioLower{ijk}{ab}(z_{ijk}) & \text{if $(a,b)=(i,j)$} \\
	     \ioLower{ijk}{ab}(z_{ijk}) & \text{if $(a,b)=(i,k)$} \\
	     -\ioLower{ijk}{ab}(z_{ijk}) & \text{if $(a,b)=(j,k)$} \\
	     0 & \text{otherwise}
	\end{matrix} \right. 
\end{align*}
 Notice that being in the kernel of $\rho$ amounts to the prelog condition. The fact that the lower left hand square commutes is proven in \cite[Prop. 2.9]{BBvG19-1}. This is the only place where the Friedman condition is used, and that part of the diagram is used only as a sanity check later. 
  
 We define the \emph{saturated numerical prelog Chow group} $\Numsatprelog^* (X)$ as the saturation of $ \Chowprelognum (X)$ in the lattice $\coker (\delta )/({\rm torsion})$.  \end{definition}
 
 Note that $\bigoplus \Chownum^*(X_{i})$ is naturally an $R(X)$-module ($R(X)$ is a subring), and also $\bigoplus \Chownum^*(X_{ij})$ is an $R(X)$-module: if $\alpha= (\alpha_i)\in R(X)$, and $z_{ij}\in \Chownum^*(X_{ij})$ we define $\alpha \cdot z_{ij} := \iota_{\{i, j\} >\{i\}}^* (\alpha_i).z_{ij} = \iota_{\{i, j\} >\{j\}}^* (\alpha_j).z_{ij}$ (the latter equality holds because of the prelog condition and shows that it does not matter of we restrict $\alpha_i$ or $\alpha_j$ to form the product). Thus every $\Chownum^*(X_{ij})$ is an $R(X)$-module, and thus so is the direct sum. 
 
 \begin{proposition}\label{pComparison}
 The map $\delta$ in the above diagram is an $R(X)$-module homomorphism, hence $\Chowprelognum (X)$ is naturally a quotient ring of $R(X)$. 
  \end{proposition}
 
 \begin{proof}
 It suffices to show that the map
 \[
 \delta\mid_{\Chownum^*(X_{ij})}\colon \Chownum^*(X_{ij}) \to \bigoplus \Chownum^*(X_{i})
 \]
 is an $R(X)$-module homomorphism. Then for a class $z=(z_i) \in \bigoplus \Chownum^* (X_i)$ and $y \in \Chownum^* (X_{ij})$ we have
\[
	z_i.(\iota_{\{i, j\} >\{i\}, *}(y)) = \iota_{\{i, j\} >\{i\}, *}(\iota_{\{i, j\} >\{i\}}^*(z_i).y)
\]
and similarly
\[
z_j.(\iota_{\{i, j\} >\{j\}, *}(y)) = \iota_{\{i, j\} >\{j\}, *}(\iota_{\{i, j\} >\{j\}}^*(z_j).y).
\]
From the way the $R(X)$-module structures on $ \Chownum^*(X_{ij})$ and $ \Chownum^*(X_{i})$ are defined, these relations mean precisely that $\Chownum^*(X_{ij}) \to \bigoplus \Chownum^*(X_{i})$ is an $R(X)$-module homomorphism. Note also that if $z=(z_i) \in R(X)$, then 
\[
\iota_{\{i, j\} >\{i\}}^*(z_i).y = \iota_{\{i, j\} >\{j\}}^*(z_j).y = z\cdot y.
\]
Thus we obtain in conclusion that 
\[
	\delta \colon  \bigoplus \Chownum(X_{ij}) \to  \bigoplus \Chownum(X_i)
\]
is an $R(X)$-module homomorphism, as desired. 
 \end{proof}

Given a strictly semistable degeneration $\pi \colon \XX \to C$ (strictly semistable=total space smooth+ central fibre reduced simple normal crossing) over some curve with marked point $t_0$ and $X \simeq \XX_{t_0}$, the specialisation homomorphism induces a natural homomorphism 
\[
\sigma_{\XX}\colon \mathrm{CH}_* (X_K) \to \Chowprelognum (X)
\]
(where $X_K$ is the generic fibre). This follows from \cite[Prop. 2.11]{BBvG19-1}.

\medskip

 If we consider a cover $C'\to C$ of smooth curves branched at $t_0$, the specialisation homomorphism $\sigma_{\XX'}$ of the pull-back family $\XX' =\XX\times_{C} C' \to C'$ (where we fix a distinguished point $t_0'$ in $C'$ mapping to $t_0$) gives a homomorphism into $ \Numsatprelog^* (X)$ by \cite[Prop. 4.2 and Prop. 4.4]{BBvG19-1}.

\section{Recollection of some formulas for Chow groups}\label{sRecollectionChow}

We now recall a few formulas for Chow rings of projective bundles and blowups needed in the sequel. For $X$ a smooth projective variety and $\EE$ a vector bundle of rank $r+1$ on $X$, we will denote by $\pi\colon \P (\EE ) \to X$ the projective bundle of lines in the fibres of $\EE$. 

\begin{proposition}\label{pChowRingProjectiveBundle}
With the previous notation let $\zeta$ be the first Chern class of the line bundle $\OO_{\P (\EE) }(1)$ in $\Chow^1(\P (\EE))$. Then as rings
\[
\Chow^*(\P (\EE ))\simeq \Chow^* (X)[\zeta]/(\zeta^{r+1} + c_{1}(\EE)\zeta^r + \dots + c_{r+1}(\EE ))
\]
where $c_i (\EE )$ are the Chern classes of $\EE$ and $\Chow^*(X)$ is considered as a subring of $\Chow^*(\P (\EE ))$ via the injective map $\pi^* \colon \Chow^*(X) \to \Chow^*(\P (\EE ))$. 
\end{proposition}

\begin{proof}
This is \cite[Thm. 9.6]{E-H16}.
\end{proof}

\begin{proposition}\label{pChowRingBlowUp}
Let $X$ be a smooth projective variety, $Z\subset X$ a smooth subvariety. Let $\pi\colon \mathrm{Bl}_Z X \to X$ be the blowup, $\NN=\NN_{Z/X}$ the normal bundle of $Z$ in $X$, $E=\P (\NN_{Z/X})\subset \mathrm{Bl}_Z X$ the exceptional divisor, and $i, j$ the natural inclusions as in the following diagram:
\[
\xymatrix{
E \ar[r]^{j}\ar[d]^{\pi_E} & \mathrm{Bl}_Z X\ar[d]^{\pi}\\
Z\ar[r]^{i} & X 
}
\]
Let $\zeta$ be the first Chern class of the line bundle $\OO_{\P (\NN_{Z/X})}(1)$ in $\Chow^1(E)$. 
Then $\Chow^* (\mathrm{Bl}_Z X )$ is generated by $\pi^*\Chow^* (X)$ and $j_* \Chow^* (E)$ as an abelian group. More precisely, there is an exact sequence of abelian groups 
\[
\xymatrix{
0 \ar[r] & \Chow^* (Z) \ar[r]^{\varphi \quad\quad\quad} & \Chow^* (E) \oplus \Chow^* (X) \ar[r]^{\quad\psi} & \Chow^* ( \mathrm{Bl}_Z X) \ar[r] & 0
}
\]
where 
\begin{gather*}
\varphi ( z) = \bigl(c_{m-1}(\QQ ) \pi_E^*(z),  -i_*(z) \bigr), \\
\psi (\gamma, \alpha ) = j_* (\gamma ) + \pi^* (\alpha )
\end{gather*}
and $m$ is the codimension of $Z$ in $X$, $\QQ$ is the universal quotient bundle on $E\simeq \P (\NN_{Z/X})$. Moreover,
\[
c_{m-1}(\QQ) = \zeta^{m-1} + c_1(\NN)\zeta^{m-2} + \dots + c_{m-1}(\NN) . 
\]
There are the following rules for multiplication:
\begin{align*}
\pi^*(\alpha)\cdot \pi^*(\beta) = \pi^*(\alpha\beta) &\quad \text{for $\alpha,\beta\in \Chow^* (X)$,}\\
\pi^*(\alpha) \cdot j_* (\gamma ) = j_* (\pi^*_Ei^* \alpha \cdot \gamma) &\quad \text{for $\alpha\in\Chow^* (X), \gamma\in \Chow^*(E)$,}\\
j_*\gamma\cdot j_*\delta = - j_* (\gamma\cdot \delta\cdot\zeta ) &\quad \text{for $\gamma, \delta\in \Chow^*(E)$.}
\end{align*}
\end{proposition}

\begin{proof}
This is \cite[Prop. 13.12, Thm. 13.14]{E-H16} and \cite[Prop. 6.7]{Ful98}.
\end{proof}

Sometimes one can find a more economical set of generators for the Chow ring of a blowup:

\begin{proposition}\label{pGettingRidOfZeta}
In the situation of Proposition \ref{pChowRingBlowUp} the ring $\Chow^*(\mathrm{Bl}_Z X)$
is generated by $\pi^*\Chow^* (X)$ and $j_* \pi_E^* \Chow^* (Z)$ as a ring.

The same statement holds with $\Chow$ replaced by $\Chownum$ everywhere.
\end{proposition}

\begin{proof}
$j_* \Chow^* (E)$ is generated as abelian group by elements of the form $j_*(\zeta^c \pi_E^* z)$. If $c \ge 1$ we have, by the last line of Proposition \ref{pChowRingBlowUp}, the following relation:
\[
	(j_* E)^c\cdot j_*(\pi_E^* z) = (-1)^c j_*(\zeta^c \pi_E^* z).
\]
We can therefore express all generators with $c \ge 1$ by such with $c=0$.

The assertion for $\Chownum$ follows from the one for $\Chow$ by passing to numerical equivalence.
\end{proof}

\begin{proposition}\label{pGettingRidOfEvenMore}
In the situation of Proposition \ref{pChowRingBlowUp} consider $\Chow^* (Z)$ as a $\Chow^* (X)$-module via the ring homomorphism $i^* \colon \Chow^* (X) \to \Chow^* (Z)$. Let
\[
	\ZZ \subset \Chow^*(Z)
\]
be a set of $\Chow^* (X)$-module generators of $\Chow^* (Z)$. Then 
$\Chow^*(\mathrm{Bl}_Z X)$ is generated as a ring by $\pi^*\Chow^* (X)$ and elements $j_* \pi_E^* z$, $z \in \ZZ$.

The same statement holds with $\Chow$ replaced by $\Chownum$ everywhere.
\end{proposition}

\begin{proof}
By Proposition \ref{pGettingRidOfZeta}, we only need to show that elements of the form $j_* \pi_E^* (y)$ with $y$ in $\Chow^* (Z)$ are in the subring generated by $\pi^*\Chow^* (X)$ and elements $j_* \pi_E^* z$, $z \in \ZZ$. 

Every such element $y$ is a sum of elements of the form  $i^*(x) \cdot z$ with $z\in \ZZ$ and $x\in \Chow^* (X)$ because $\ZZ$ is a set of $\Chow^* (X)$-module generators for $\Chow^* (Z)$. But
\[
j_* \Bigl( \pi_E^*\bigl(  i^*(x) \cdot z \bigr) \Bigr) 
= j_* \Bigl( \pi^*_E\bigl( i^*(x) \bigr) \cdot \pi_E^*\bigl( z ) \Bigr) 
= \pi^*(x) \cdot j_* \pi_E^*\bigl( z ) .
\]
The same proof goes through with $\Chow$ replaced by $\Chownum$ everywhere.
\end{proof}



We also need a few facts about Chow groups and rings of products. The first result says when a K\"unneth formula can be expected to hold.

In the following, by a \emph{linear variety} we mean a variety in the class of varieties constructed by an inductive procedure starting with an affine space of any dimension, in such a way that the complement of a linear variety imbedded in affine space in any way is a linear variety, and a variety stratified into a finite disjoint union of linear varieties is a linear variety. This is what is called a linear scheme in \cite{To14}. 

\begin{proposition}\label{pTotaroLinearSchemes}
Let $X$ be a linear variety.Then for any variety $Y$
\[
\Chow_* (X\times Y) \simeq \Chow_* (X) \otimes_{\Z} \Chow_* (Y).
\]
\end{proposition}
\begin{proof}
This is \cite[Prop. 1]{To14}.
\end{proof}

The next concerns Chow groups, modulo numerical equivalence, for self-products of very general curves of genus $\ge 1$. We will use this Proposition in the proof of  Lemma \ref{pCxCdsg} below. 

\begin{proposition}\label{pCxC}
Let $C$ be a very general smooth projective curve of genus $g$. Then $\Chownum^1 (C\times C)$ is generated by the class of the diagonal $[\Delta_C]$ and by $[C\times \{p\}]$ and $[\{p\}\times C]$ where $p\in C$ is a point.
\end{proposition}

\begin{proof}
Recall \cite[\S 11.5]{BL04} that a correspondence between smooth algebraic curves $C_1$, $C_2$ is a line bundle $L$ on $C_1\times C_2$, and two correspondences $L, L'$ are said to be equivalent if $L' = L \otimes \mathrm{pr}_1^* L_1 \otimes \mathrm{pr}_2^* L_2$ for some line bundles $L_1, L_2$ on $C_1, C_2$ and $\mathrm{pr}_i \colon C_1\times C_2\to C_i$ the projections. Then \cite[Thm. 11.5.1]{BL04} shows that, denoting by $\mathrm{Corr}(C_1, C_2)$ the $\Z$-module of equivalence classes of correspondences between $C_1$ and $C_2$, there is a canonical isomorphism of abelian groups
\[
\mathrm{Corr}(C_1, C_2) \to \mathrm{Hom}(\mathrm{Jac}(C_1), \mathrm{Jac}(C_2)), 
\]
where $\mathrm{Jac}(C_i)$ is the Jacobian of $C_i$, and it follows from the definition of this isomorphism in loc. cit. that for $C=C_1=C_2$ it maps the line bundle $\OO (\Delta_C)$ to the identity in $\mathrm{End} (\mathrm{Jac}(C))$. 

Thus to conclude the proof of the Proposition, it suffices to remark that for $C$ a very general curve of genus $g$, we have $\mathrm{End} (\mathrm{Jac}(C))\simeq \Z$, where $\Z$ is generated by the image of the identity: this is stated in \cite[\S 11.12.13]{BL04} and proved in \cite{Koi76}. Note that in \cite{BL04} it says that the result holds for ``a general curve" whereas in \cite{Koi76} the author refers to a ``generic point of the moduli space": both wordings are somewhat imprecise from the viewpoint of modern terminology; what is actually true and proved is that the result is valid outside a countable union of proper closed subvarieties of the moduli space, or for a ``very general curve". Already the case of elliptic curves shows that the restriction to very general curves is necessary: there are infinitely many non-isomorphic elliptic curves with complex multiplication (for example, the class numbers associated with imaginary quadratic number fields are not bounded).
\end{proof}

\section{Numerical Chow rings via dual socle generators}\label{sDualSocle}

From now on we work with Chow rings modulo numerical equivalence, and denote these by $\Chownum^* (X)$ for a smooth projective variety $X$. We would like to be able to write these numerical Chow rings, which are Artin rings, in more compact and computationally convenient form. For this we briefly recall some facts about zero-dimensional Gorenstein rings from \cite[Section 21.2]{Ei04}, partly to set up notation. 

\medskip

Let $k$ be a field (later we will work with a subring, too, for us $k=\Q$ and the subring will be $\Z$), and let 
\[
R= k[x_1, \dots , x_r], \quad R^* = k[x_1^{-1}, \dots , x_r^{-1}]
\]
be the polynomial rings in variables $x_i$ and their inverses, respectively, both considered as subrings of $K=k(x_1, \dots , x_r)$. We make $R^*$ into an $R$-module by decreeing that for monomials $m\in R$ and $n\in R^*$, $m\cdot n$ is to be the product $mn \in K$ if this lies in the subring $R^*$, and zero otherwise. Now \cite[Thm. 21.6]{Ei04} says that the ideals $I \subset (x_1, \dots , x_r)$ such that $R/I$ is a local zero-dimensional Gorenstein ring are precisely the ideals of the form $I= \mathrm{Ann}_R (f)$ for some nonzero element $f\in R^*$. Here $f$ is called the \emph{dual socle generator} of $R/I$. 

\begin{notation}\label{nDualSocle}
Let $X$ be a smooth projective variety of dimension $d$ and let $\Chownum^* (X)$ and $\Chownum^* (X)_{\Q}:=\Chownum^* (X)\otimes_{\Z}\Q$ be its Chow ring of cycles modulo numerical equivalence with $\Z$ and $\Q$ coefficients, respectively.  Note that $\Chownum^* (X)$ is graded by the codimension of the cycles involved. 
Let $x_1, \dots , x_r$ be variables corresponding to homogeneous generators $\alpha_1, \dots , \alpha_r$ of $\Chownum^* (X)$ with weight $\weight (x_i) =\mathrm{codim}\, \alpha_i$.

We use multi-index notation and write:
\begin{align*}
a=& (a_1, \dots , a_r)\in \N^r\\
|a|=& \sum_{i=1}^r \weight (x_i)a_i \\
\alpha^a =& \alpha_1^{a_1}\cdot \ldots \cdot \alpha_r^{a_r} \in \Chownum^* (X)\\
x^{-a}=& x_1^{-a_1}\cdot \ldots \cdot x_r^{-a_r} \in \Z [x_1^{-1}, \dots , x_r^{-1}]
\end{align*}
\end{notation}

\noindent Notice that the degree map
\[
\deg \colon \Chownum^d (X) \to \Z
\]
sending the class of a point to $1$ is an isomorphism. 

There is a natural surjection of graded rings
\[
\xi \colon \Z [x_1, \dots , x_r] \to \Chownum^* (X)
\]
sending $x_i$ to $\alpha_i$, and similarly with $\Q$-coefficients. 

\begin{lemma}\label{lChowNumDualSocle}
With the previous notation, let 
\[
f_X = \sum_{|a| =d} \deg (\alpha^a) \cdot x^{-a} \in \Z [x_1^{-1}, \dots , x_r^{-1}].
\]
Then
\[
\Chownum^* (X) = \Z [x_1, \dots , x_r] / \mathrm{Ann}(f_X), \quad \Chownum^* (X)_{\Q} = \Q [x_1, \dots , x_r] / \mathrm{Ann}(f_X) .
\]
In particular, $\Chownum^* (X)_{\Q}$ is a Gorenstein local ring with dual socle generator $f_X$. 
\end{lemma}

\begin{proof}
First notice that if $p\in \Z [x_1, \dots , x_r]$ is homogeneous of weight $d$, then $p\cdot f_X =\deg (\xi (p))$ by definition of the pairing and of $f_X$. 

We need to show that $\ker (\xi )=\mathrm{Ann}(f_X)$. 

Let $p\in \ker (\xi )$. We can assume without loss of generality that $p$ is homogeneous of weight $\delta$ since $\xi$ is graded. Then $\xi (p) \in \Chownum^\delta (X)$ is zero, hence for every $q \in \Z [x_1, \dots , x_r]$ homogeneous of weight $d-\delta$ we have 
\[
pq f_X = \deg (\xi (p\cdot q)) = \deg(\xi (p) \xi (q)) = 0 .
\]
In other words, every such $q$ annihilates $pf_X$. Since the pairing $R_{d-\delta}\times R^*_{-(d-\delta)} \to \Z$ is perfect (or, becomes perfect over $\Q$), we have $pf_X=0$. Hence $p\in \mathrm{Ann}(f_X)$.

Conversely, if $p$ is in $\mathrm{Ann}(f_X)$ and homogeneous of weight $\delta$, let $\eta \in \Chownum^{d-\delta} (X)$ be arbitrary and write it as $\xi (q) =\eta$ with $q$ homogeneous of weight $d-\delta$. Then we have
\[
\deg (\eta . \xi (p)) = \deg (\xi (q) .\xi (p))  =\deg (\xi (qp) ) = qp f_X = 0. 
\]
Hence $\xi (p)$ is zero in $\Chownum^* (X)$ by definition of numerical equivalence. 

The same proof works with $\Q$-coefficients. 
\end{proof}

\begin{lemma}\label{lEnoughRelations}
Let $X$ be a smooth projective variety of dimension $d$. Let
\[
\xi \colon \Z [x_1, \dots , x_r] \to \Chownum^* (X)
\]
be a surjection of graded rings as above. Let $I$ be an ideal contained in $\ker \xi$ such that the induced homomorphism
\[
\overline{\xi} \colon \Z [x_1, \dots , x_r]/I \to \Chownum^* (X)
\]
is an isomorphism between elements of weight $d$:
\[
\overline{\xi} \colon (\Z [x_1, \dots , x_r]/I)_d \to \Chownum^* (X)_d \simeq \Z .
\] 
Then
\[
f_X = \sum_{|a| =d} \overline{\xi} (x^a) \cdot x^{-a} \in \Z [x_1^{-1}, \dots , x_r^{-1}]
\]
and we can write
\[
\Chownum^* (X) = \Z [x_1, \dots , x_r] / \mathrm{Ann}(f_X). 
\]
In other words, $I$ is enough to compute the entire kernel of $\xi$.
\end{lemma}

\begin{proof}
Since $\overline{\xi}$ induces an isomorphism between elements of weight $d$ by assumption, we get 
\[
\overline{\xi} (x^a) = \deg (\alpha^a)
\]
for $|a|=d$. The claim then follows from Lemma \ref{lChowNumDualSocle}. 
\end{proof}

\section{The case of cubic threefolds}\label{sCubicThreefolds}

Consider a degeneration $\pi_{\VV}\colon \VV \to B$ of smooth cubic threefolds, over a Zariski open subset $B\subset \mathbb{A}^1$ containing $0\in \mathbb{A}^1$, into the union of a smooth quadric and a hyperplane in $\P^4$, given by an equation
\[
\{ lq - tf =0\} \subset \P^4\times B
\]
where $l, q, f \in \C[X_0, \dots , X_4]$ are homogeneous of degree $1, 2, 3$, respectively, and 
\begin{enumerate}
\item
$q$ defines a nonsingular quadric $Q$;
\item
$f$ is general, in particular $f=0$ defines a smooth cubic threefold $V$;
\item
the hyperplane $L$ defined by $l$ in $\P^4$ intersects $Q$ transversely in a smooth quadric surface $S\simeq\P^1\times\P^1$;
\item
$S \cap V$ is a smooth divisor $C$ of bidegree $(3,3)$ in $S$, which is a genus $4$ canonical curve in $L\simeq\P^3$. It is well-known that any nonhyperelliptic genus $4$ curve arises as a complete intersection of a quadric and a cubic in $\P^3$, hence if we choose $f$ (very) general and $l$, $q$ such that $l=q=0$ is a smooth quadric (all of these are projectively equivalent), the curve $C$ will be a (very) general canonical curve of genus $4$.
\end{enumerate}

Notice that the total space $\VV$ is singular in $C$. After shrinking $B$, we can assume there are no singularities outside the fibre over $0$.

We blow up the non-Cartier divisor $L$ in the total space $\VV$ and get a strictly semistable degeneration $\pi_{\XX}\colon \XX \to B$ with central fibre 
\[
X=L_C\cup Q.
\]
Here, by slight abuse of notation, we denote by $Q$ the irreducible component of $X$ mapping isomorphically to $Q$ in $\VV$ under the natural morphism $\XX\to \VV$. $L_C$ is the blowup of $L$ in $C$ with exceptional divisor $E_C$. $L_C$ and $Q$ intersect in a surface which is naturally isomorphic to $S$: in $Q$ we have the previous copy of $S$, and in $L_C$ the strict transform of $S$. Hence we denote this new surface by $S$ as well.

\begin{lemma} \label{lNumLC}
Let $H\in \Chownum^1 (L_C)$ be the pullback of a hyperplane class in $L=\P^3$,
$E = \P(\NN_{C/\P^3})$ the class of the exceptional divisor in $L_C$ and $F := \pi^*_E(P)$ with $P$ a point on $C$ (the class of a fiber of $E$).
Then 
\[
\Chownum^* (L_C) = \Z[H, E, F]/ \mathrm{Ann}(f_{L_C})
\]
with 
$$f_{L_C}= H^{-3}-6H^{-1}E^{-2}-30E^{-3}-E^{-1}F^{-1}.$$
\end{lemma}

\begin{proof}
The Chow ring of $\P^3$ blown up in a smooth curve $C$ (for rational equivalence) is calculated in \cite[Prop. 13.13]{E-H16}. From this it follows that $H, E, F$ are ring generators of $\Chownum^* (L_C)$ and the intersection numbers in the dual socle generator as defined in Lemma \ref{lChowNumDualSocle} are the ones given above. 
\end{proof}

\begin{lemma} \label{lNumQ}
Let $S\in \Chownum^1 (Q)$ be the class of a hyperplane section of $Q$ and let $L$ be 
the class of a line in $Q$.
Then 
\[
\Chownum^* (Q) = \Z[S,L]/\mathrm{Ann}(f_Q) 
\]
with 
$$f_{Q}= 2S^{-3}+S^{-1}L^{-1}.$$
\end{lemma}

\begin{proof}
Apply \cite[Example 1.9.1]{Ful98}: here we use the Bruhat stratification of the rational homogeneous variety $Q$, a paving in affine spaces; thus the closures of the strata are 
\[
Q \supset T_pQ\cap Q \supset l \supset p
\]
where $T_pQ\cap Q$ is a quadric cone over a smooth conic (rationally equivalent to $S$), $l$ a line in that cone, $p\in l$ a point.
\end{proof}

\begin{lemma} \label{lNumS}
Let $S=\P^1\times \P^1$ as above. Let $R_1, R_2$ be the classes of the two rulings.  
Then
\[
\Chownum^* (S) = \Z[R_1, R_2]/\mathrm{Ann}(f_{S})
\]
with 
$$f_{S}= R_1^{-1}R_2^{-1}.$$
\end{lemma}

\begin{proof}
This is obvious. 
\end{proof}

\begin{lemma}\label{lRestrictionsAmongChow}
Let $\iota_{S, L_C}\colon S\to L_C$ and $\iota_{S, Q}\colon S \to Q$ be the natural inclusions. 
Then 
\begin{align*}
\iota_{S, L_C}^* ( H,\: E,\: F) &= \bigl(  R_1+R_2,\: 3(R_1 +R_2) ,\: R_1R_2 \bigr), \\
\iota_{S, Q}^* ( S,\: L) &= \bigl( R_1+R_2,\: R_1R_2 \bigr) 
\end{align*}
and
\begin{align*}
(\iota_{S, L_C})_* ( 1,\: R_1,\: R_2,\: R_1R_2 ) &= \bigl(  2H-E,\:  H^2 -3F,\: H^2-3F,\: H^3 \bigr), \\
(\iota_{S, Q})_*  ( 1,\: R_1,\: R_2,\: R_1R_2 ) &= \bigl( S,\: L,\: L,\: SL \bigr) .
\end{align*}

\end{lemma}

\begin{proof}
Only the third formula is not obvious. For the third formula, remark that the class of the strict transform of $S$ in $L_C$ is given by $2H - E$, and that the lines of the rulings on $S$ are trisecants to the bidegree $(3,3)$ curve $C$, hence give classes $H^2-3F$ as claimed. 
\end{proof}

\

\section{The product family}\label{sProductFamily}

Consider the product family $\XX\times_B \XX \to B$. The total space is singular in a variety isomorphic to $S\times S$ contained in the central fibre as the locus where all four irreducible components $L_C\times L_C$, $L_C\times Q$, $Q\times L_C$, $Q\times Q$ intersect. We now blow up $L_C\times Q$ in the total space and obtain a strictly semistable degeneration $\pi_{\YY}\colon \YY\to B$ with components of the central fibre $Y$ given by 
\begin{align*}
Y_1=& L_C\times L_C,\\
Y_2=& \mathrm{Bl}_{S\times S} \left(  L_C\times Q  \right), \\
Y_3=& \mathrm{Bl}_{S\times S} \left(  Q\times L_C  \right), \\
Y_4=&Q\times Q.
\end{align*}
In Figure \ref{ProductConfiguration} we have indicated these four components and their mutual intersections.

The fact that $\pi_{\YY}\colon \YY\to B$ is strictly semistable after this one blowup can be checked by a local calculation: $\XX\times_B \XX$ is singular in the points where all four irreducible components of the central fibre intersect and locally around such a point $\XX\times_B \XX$ is given by equations $xy-t=0, \: ab-t=0$ and the blowup centre is given by $x=a=t=0$. Eliminating $t$ we need to blow up $xy-ab=0$, which, up to taking a product with an affine space, is the cone over a smooth quadric surface in $\P^3$. We then blow up the locus $x=a=0$, which is a smooth surface through the vertex of the cone that maps to a line on the quadric. We obtain a small resolution of the vertex of the cone in this way.

\begin{figure}[h]
	\centering
	\includegraphics[scale=0.45]{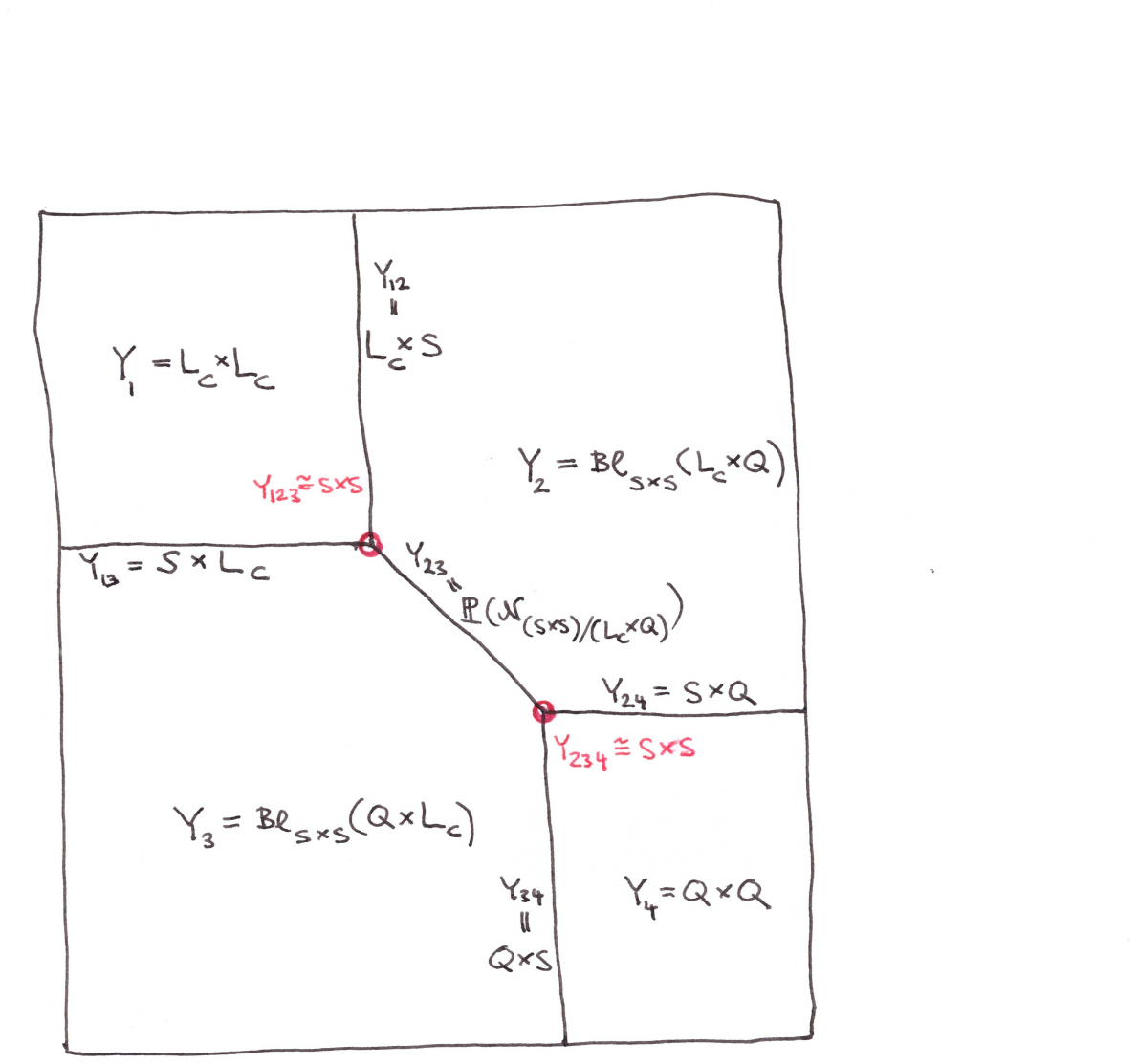}
	\caption{}\label{ProductConfiguration}
\end{figure}

The mutual intersections $Y_{ij}=Y_i\cap Y_j$ of these components are:
\begin{align*}
Y_{12}=& L_C\times S\\
Y_{13}=& S\times L_C \\
Y_{14}=& \emptyset \\
Y_{23}=& \P \left( \NN_{(S\times S)/(L_C\times Q)} \right) \simeq \P \left( \NN_{(S\times S)/(Q\times L_C)} \right)  \\
Y_{24}=&S \times Q \\
Y_{34}=& Q\times S
\end{align*}

The inclusions 
\begin{align*}
\iota_{\{1,2\} , \{1\}}\colon & Y_{12} \to Y_1\\
\iota_{\{1,3\} , \{1\}}\colon & Y_{13} \to Y_1\\
\iota_{\{2,4\} , \{4\}}\colon & Y_{24} \to Y_4\\
\iota_{\{3,4\} , \{4\}}\colon & Y_{34} \to Y_4
\end{align*}
are products of the natural inclusions $\iota_{S, L_C}\colon S\to L_C$ and $\iota_{S, Q}\colon S \to Q$ with identity maps. The inclusions of $Y_{23}$ into $Y_2$ and $Y_3$ are the inclusions of the exceptional divisors of the respective blowups.
The inclusion of $Y_{13}$ into $Y_3$ is obtained as follows: one has the inclusions
\[
S\times S \subset S \times L_C \subset Q\times L_C. 
\] 
Thus we see that, blowing up the $S\times S$ in $Q\times L_C$, the strict transform of $S\times L_C$ is isomorphic to $S\times L_C$. This isomorphism composed with the inclusion into $Y_3$ gives $\iota_{\{1,3\} , \{3\}}$. Similarly for the remaining cases.

The triple intersections are
\begin{align*}
Y_{123}=Y_{234}= & S\times S.
\end{align*}

The inclusions $\iota_{\{1,2,3\} ,\{i,j\}}$ are clear unless $\{ i,j\} =\{2,3\}$ in which case we deal with it in Proposition \ref{pPushPull23}.

We seek to compute the saturated prelog Chow group of the central fibre $Y$ of $\YY\to B$. 

\subsection{The numerical Chow ring of $Y_1$}

For this we begin with $\Chownum^*(Y_1)$. Notice that $Y_1 = L_C \times L_C$ is the blowup of
$L_C \times \P^3$ in $L_C \times C$, and that  furthermore $L_C \times C$ is the blowup of $\P^3 \times C$ in $C \times C$. This gives the following combined blowup diagram:

\[
\xymatrix{
& E_{L_C \times C} \ar[d]^{\pi'_E}  \ar@{^{(}->}[r]^{j'}  
& L_C\times L_C\ar[d]^{\pi'} 
\\
E_{C \times C} \ar[d]^{\pi_E}\ar@{^{(}->}[r]^j 
& L_C\times C\ar[d]^{\pi} \ar@{^{(}->}[r]^{i'}  
& L_C\times\P^3 
\\
C\times C \ar@{^{(}->}[r]^i 
& \P^3\times C &
}
\]

\begin{lemma}\label{pCxCdsg}
We have
\[
\Chownum^* (C\times C) = \Z [p, P, \Delta_C]/\mathrm{Ann}(f_{C\times C})
\]
where $p$ is the pullback of the class of a point on $C$ via the first projection, $P$ the same via the second projection, and $\Delta_C$ the class of the diagonal and
\begin{gather*}
f_{C\times C}= p^{-1}P^{-1} + p^{-1}\Delta_C^{-1} + \Delta_C^{-1}P^{-1} - 6\Delta_C^{-2}.
\end{gather*}
\end{lemma}

\begin{proof}
We use Proposition \ref{pCxC} and the fact that the normal bundle of the diagonal in $C\times C$ is the tangent bundle $T_C$ of $C$: therefore, 
\[
\Delta_C^2 =\deg T_C = 2 - 2g = -6. 
\]
\end{proof}

\begin{proposition} \label{pGensLCxC}
$\Chownum^*(L_C \times C)$ is generated as a ring by the elements
\[
	\{h,e,f,\bar{P},\overline{\Delta}_C \},
\]
where $h,e,f$ are the classes coming from $L_C$ and $\bar{P}$ the point class coming from the second factor of $L_C \times C$. Furthermore $\overline{\Delta}_C = j_*\pi_E^*(\Delta_C)$ is the $\P^1$ bundle over $\Delta_C$ in $E_{C\times C}$.\end{proposition}

\begin{proof}
By Lemma \ref{pCxCdsg}, $\Chownum^*(C \times C)$ is generated as an abelian group by
\[
	\{ 1, p, P, \Delta_C, P \Delta_C \} .
\]
Let $\widetilde{P} \in \Chownum^* (\P^3\times C)$ be the pullback of the class of a point on $C$ via the second projection $\P^3\times C \to C$. 
Since $P = i^*(\widetilde{P})$, $\Chownum^*(C \times C)$ is generated by
\[
	\{ 1,p,\Delta_C \}
\]
as a $\Chownum^*(\P^3 \times C)$-module. Noticing that
\begin{align*}
	j_*\pi_E^*(1) &= e, \\
	j_*\pi_E^*(p) & = f, \\
	j_*\pi_E^*(\Delta_C) & = \overline{\Delta}_C
\end{align*}	
and that $\Chownum^*(\P^3 \times C)$ is generated by $\{h,\widetilde{P}\}$ as a ring and $\pi^* (\widetilde{P})=\bar{P}$, we see that Proposition \ref{pGettingRidOfEvenMore} then gives the claim.
\end{proof}

\begin{proposition} \label{pGensLCxLC}
We have that
\[
 	\Chownum^* (L_C \times L_C)
\]
is generated as a ring by
\[
	\{h,e,f,H,E,F,D \},
\]
where $h, e, f, H, E, F$ are the classes as in Lemma \ref{lNumLC} coming from the two factors, and $D:=j'_* (\pi_E')^* \overline{ \Delta}_C$ is the class of the $\P^1\times\P^1$-bundle over  the diagonal in $C\times C$. 
\end{proposition}

\begin{proof}
Since $e,f,h$ are in the image of $(i')^*$ we have that $\Chownum^*(L_C \times C)$ is generated by
\[
	\{1,\overline{P},\overline{\Delta}_C,\overline{P}\, \overline{\Delta}_C \}
\]
as a $\Chownum^* (L_C \times \P^3)$-module. Furthermore 
\[
	\bar{P}\overline{\Delta}_C 
	= \bar{P}j_*\pi_E^*(\Delta_C) 
	= j_*\pi_E^*( P\Delta_C)
	= j_*\pi_E^*(Pp)
	= \bar{P}j_*\pi_E^*(p)
	=\bar{P}f
\]
so that $\Chownum^*(L_C \times C)$ is already generated by
\[
	\{1,\overline{P},\overline{\Delta_C} \}
\]
as a $\Chownum^* (L_C \times \P^3)$-module. Noticing that
\begin{align*}
	j'_* (\pi'_E)^* (1)  &= E\\
	j'_* (\pi'_E)^* (\bar{P}) &= F\\
	 j'_* (\pi'_E)^*(\overline{\Delta}_C) &=j'_* (\pi'_E)^* j_*\pi_E^*(\Delta_C) = D
\end{align*}
and that $\Chownum^*(L_C \times \P^3)$ is generated by $\{h,e,f,H\}$ as a ring, we see that the claim follows from Proposition \ref{pGettingRidOfEvenMore}.
\end{proof}

\begin{proposition} \label{pLCxLCdsg}
With the notation of the previous Proposition we have 
\[
	\Chownum^* (Y_1) 
	= \Chownum^* (L_C\times L_C) 
	= \Z[h, e, f, H, E, F, D]/ \mathrm{Ann}(f_{L_C\times L_C}),
\]
where
\small
\begin{align*}
f_{L_C\times L_C} = &f_{L_C} \cdot f'_{L_C} \\
 & + D(30e^{-2}E^{-1} + 30e^{-1}E^{-2}  + 6e^{-1}E^{-1}h^{-1} + 6e^{-1}E^{-1}H^{-1} + E^{-1}f^{-1} + e^{-1}F^{-1})\\
 &-6D^{-2}.
\end{align*}
\normalsize
\end{proposition}

\begin{proof}
We have to calculate the intersection numbers of all monomials of degree $6$ in the generators. Here $\{H,E,h,e\}$ have degree $1$, $\{F,f\}$ have degree $2$ and $D$ has degree $3$. We can write every such polynomial as $mMD^c$ where $m$ is a monomial in the generators of $\Chownum^*(L_C)$ of the first factor (lower case letters) and $M$ is a monomial in the generators of $\Chownum^*(L_C)$ of the second factor (upper case letters). We then have the following cases
\begin{itemize}
\item[$c=0:$] Here the monomial is $mM$ and the intersection number of the product is the product
of the intersection numbers on the first and second $L_C$ respectively. These intersection numbers are calculated by the dual socle generators of the factors and we obtain the summand
\[
	f_{L_C} \cdot f'_{L_C}
\]
of $f_{L_C \times L_C}$.
\item [$c=1:$] We can calculate these intersection numbers on $D$. We recall that $D$ is a $\P^1 \times \P^1$ bundle on $C$, where we identify $C$ with the diagonal of $C \times C$. The intersection ring
$\Chownum^*(D)$ is generated by the pullback of a point $P$ on $D$, and the relative hyperplane class $\gamma$ and $\Gamma$ of the first and second factor. We have $\gamma^2 = -30\gamma P$ and
$\Gamma^2 = -30 \Gamma P$. Let $i$ be the inclusion of $D$ in $L_C \times L_C$. The pullbacks of the generators to $D$ are:
\begin{align*}
	i^*(h) &= 6P &   i^*(H) &= 6P \\ 
	i^*(e) &= -\gamma & i^*(E) &= -\Gamma \\
	i^*(f) &= -\gamma P & i^*(F) &= -\Gamma P
\end{align*}
The class of a point in $D$ is $\gamma\Gamma P$ so the non-zero intersection numbers are collected
in the summand \small
\[
	D(30e^{-2}E^{-1} + 30e^{-1}E^{-2}  + 6e^{-1}E^{-1}h^{-1} + 6e^{-1}E^{-1}H^{-1} + E^{-1}f^{-1} + e^{-1}F^{-1}) .
\] \normalsize
\item [$c=2:$] We compute
\begin{align*}
	D^2 
	&= \bigl(j'_* (\pi_E')^* j_* \pi_E^* \Delta_C \bigr)^2 \\	
	&= -j'_*\Bigl(\Gamma 
	\bigl((\pi_E')^* j_* \pi_E^* \Delta_C \bigr)^2
	\Bigr)  \\	
	&= -j'_*\Bigl(\Gamma 
	(\pi_E')^* \bigl(j_* \pi_E^* \Delta_C \bigr)^2
	\Bigr)  \\	
	&= j'_*\Bigl(\Gamma 
	(\pi_E')^* j_* \bigl( \gamma (\pi_E^* \Delta_C)^2
	\bigr)\Bigr)  \\	
	&= j'_*\Bigl(\Gamma 
	(\pi_E')^* j_* \bigl( \gamma \pi_E^* ( \Delta_C^2 )
	\bigr)\Bigr)  .
\end{align*}
Since $\Delta_C^2 = -6$, this proves that also $D^2 = -6$. 
\end{itemize}
\end{proof}

\subsection{The numerical Chow ring of $Y_2$}

\noindent
Next we turn to $Y_2$, the blowup of $L_C \times Q$ in $S \times S$, as in the
blowup diagram:
\[
\xymatrix{
N \ar[r]^{j}\ar[d]^{\pi_N} & Y_2\ar[d]^{\pi}\\
S \times S \ar[r]^{i} & L_C \times Q
}
\]

\begin{lemma}\label{lNumN}
We have
\[
\Chownum^*(N) = \Chownum^*(S\times S) [\xi]/ \bigl( (\xi - r_1-r_2)(\xi +R_1+R_2) \bigr) .
\]
Here $\xi$ is the relative hyperplane class of the projectivisation of $\NN_{(S\times S)/ (L_C\times Q)}$, which is naturally isomorphic to $N$.
\end{lemma}

\begin{proof}
We have 
\[
\NN_{(S\times S)/ (L_C\times Q)} \simeq \OO_{(\P^1\times\P^1)\times (\P^1\times\P^1)} (-1,-1, 0, 0) \oplus \OO_{(\P^1\times\P^1)\times (\P^1\times\P^1)} (0,0,1,1).
\]
The assertion then follows from Proposition \ref{pChowRingProjectiveBundle}. 
\end{proof}

\begin{proposition}\label{pGeneratorsY2}
Let $\sheaf{X}$ be a basis, consisting of homogeneous elements, of the free $\Z$-module $\Chownum^*(N)$. Consider the homomorphism of graded rings 
\[
\Phi\colon \Z [h,e,f, S, L, N_{x}]_{x\in \sheaf{X}} \to \Chownum^*(Y_2)
\]
where the variables are mapped to the corresponding classes in $\Chownum^*(Y_2)$, where the class corresponding to $N_x$ is $j_* (x)$. 
For an arbitrary
\[
x =\sum_i n_i x_i \in \Chownum^* (N), \quad x_i \in \sheaf{X}, n_i \in \Z
\]
we set
\[
N_x := \sum_i n_i N_{x_i} \in \Z [h,e,f, S, L, N_{x}]. 
\]
Let $\sheaf{Y}=\{ h,e,f, S, L \}$.  Let $\ZZ$ be the basis of the $\Z$-module $\Chownum^* (S\times S)$ consisting of all nonzero monomials in $r_1, r_2, R_1, R_2$. For $z\in \ZZ$ we choose an element $\overline{z}\in \Z [h,e,f, S, L, N_{x}]$  such that $\Phi (\overline{z}) = \pi^* i_* (z)$ using Lemma \ref{lRestrictionsAmongChow}.

Let $I$ be the ideal in $\Z [h,e,f, S, L, N_{x}]$ generated by the elements
\begin{align*}
\alpha \cdot N_{\gamma } - N_{(j^* \Phi (\alpha)) . \gamma} &\quad \text{for $\alpha\in\sheaf{Y}, \gamma\in \sheaf{X}$,}\\
N_{\gamma} \cdot N_{\delta} + N_{\gamma \cdot \delta\cdot\xi}  &\quad \text{for $\gamma, \delta\in \sheaf{X}$},\\
N_{c_{1}(\QQ ) \pi_N^*(z)}  - \overline{z} & \quad \text{for $z\in \ZZ$}
\end{align*}
where
\[
c_1 (\QQ ) = \xi + (-r_1-r_2+R_1+R_2).
\]
In other words, $I$ is the ideal of relations from Proposition \ref{pChowRingBlowUp} with $\zeta=\xi$.

Then we have
\[
\Chownum^* (Y_2) = \Z [h,e,f, S, L, N_{x}]/I .
\]
\end{proposition}

\begin{proof}
We want to apply Lemma \ref{lEnoughRelations}. It follows from Proposition \ref{pChowRingBlowUp} that $I \subset \ker \Phi$. A Macaulay2 computation done in \cite[Proposition 6.6 in \textsf{prelogYparts.m2}]{BBG-M2} shows that $(\Z [h,e,f, S, L, N_{x}]/I )_6$ is a free $\Z$-module of rank $1$ generated by an element that maps to a class of a point under $\Phi$. Lemma \ref{lEnoughRelations} then gives us a formula for $f_{Y_2}$ such that $\ker \Phi = \mathrm{Ann} (f_{Y_2})$. A Macaulay2 computation done in \cite[Proposition 6.6 in \textsf{prelogYparts.m2}]{BBG-M2} then shows that $\mathrm{Ann} (f_{Y_2})=I$.
\end{proof}

\subsection{The numerical Chow ring of $Y_3$}

\noindent
Next we turn to $Y_3$, the blowup of $Q\times L_C$ in $S \times S$, as in the
blowup diagram:
\[
\xymatrix{
M \ar[r]^{j}\ar[d]^{\pi_M} & Y_3\ar[d]^{\pi}\\
S \times S \ar[r]^{i} &   Q\times  L_C 
}
\]

\begin{lemma}\label{lCompareMN}
Let $\eta$ be the relative hyperplane class of the projectivisation of $\NN_{(S\times S)/ (Q\times L_C)}$, which is naturally isomorphic to $M$. 
Then there is a unique isomorphism of rings 
\[
\mu^*\colon \Chownum^* M \to \Chownum^* N, 
\]
mapping $\eta$ to $\xi -r_1-r_2+R_1+R_2$ and making the following diagram commutative:
\[
\xymatrix{
\Chownum^* M \ar[rr]^{\mu^*} & & \Chownum^* N\\
 & \Chownum^* (S\times S)\ar[lu]^{\pi^*_M}\ar[ru]_{\pi^*_N} & 
}
\]
\end{lemma}

\begin{proof}
We have 
\[
\EE_N:=\NN_{(S\times S)/ (L_C\times Q)} \simeq \OO_{(\P^1\times\P^1)\times (\P^1\times\P^1)} (-1,-1, 0, 0) \oplus \OO_{(\P^1\times\P^1)\times (\P^1\times\P^1)} (0,0,1,1)
\]
and
\[
\EE_M:=\NN_{(S\times S)/ (Q\times L_C)} \simeq \OO_{(\P^1\times\P^1)\times (\P^1\times\P^1)} (1,1, 0, 0) \oplus \OO_{(\P^1\times\P^1)\times (\P^1\times\P^1)} (0,0,-1,-1).
\]
Hence 
\[
(\ast) \quad \EE_N \otimes \LL \simeq \EE_M
\]
with $\LL = \OO_{(\P^1\times\P^1)\times (\P^1\times\P^1)} (1,1, -1, -1)$. Observe that by definition $M = \P (\EE_M)$, $N=\P (\EE_N)$. Let
\[
\mu\colon N \to M
\]
be the isomorphism induced by $(\ast )$.

Then 
\begin{align*}
(\pi_{M})_* (\OO_M (\eta )) &= \EE_M^{\vee} \\
&= \EE_N^{\vee}\otimes \LL^{\vee}\\
&  =(\pi_N)_* (\OO_N(\xi )) \otimes \LL^{\vee} \\
&  =(\pi_M\circ \mu)_* (\OO_N(\xi)) \otimes \LL^{\vee} \\
&= (\pi_M)_*\bigl( \mu_* (\OO_N(\xi )) \otimes \pi_M^* (\LL^{\vee} )\bigr).
\end{align*}

On the other hand, we know from the structure of the Picard group of a projective bundle that
\[
\OO_M (\eta ) \simeq \mu_* (\OO_N(\xi )) \otimes \pi_M^* ((\LL')^{\vee})
\]
for some line bundle $\LL'$ on $S\times S$ since $\OO_M (\eta )$ and $\mu_* (\OO_N(\xi))$ restrict to $\OO (1)$ on every fibre of $\pi_M$. Therefore both $\LL$ and $\LL'$ are line bundles on $S\times S$ with
\[
\EE_N \otimes \LL \simeq \EE_M, \quad \EE_N \otimes \LL' \simeq \EE_M
\]
and passing to the determinants on both sides we get $\LL^{\otimes 2}=(\LL')^{\otimes 2}$ which shows $\LL \simeq \LL'$ since $\mathrm{Pic} (S\times S)$ is torsionfree. 
Hence 
\[
\eta =\mu_*( \xi )- r_1-r_2+R_1+R_2.
\]
Now using $\mu_* =(\mu^{-1})^*$ gives the result.
\end{proof}

\begin{proposition}\label{pGeneratorsY3}
We have
\[
\Chownum^* (Y_3) = \Z [s, l, H, E, F, M_{x}]/J
\]
where $x$ runs over a generating set of $\Chownum^*(N)$ as a $\Z$-module, $M_x :=j_* \circ \mu_* (x)$, and $J$ is the ideal of relations derived from Proposition \ref{pChowRingBlowUp} with $\zeta=\xi - r_1-r_2+R_1+R_2$ in the same way as in Proposition \ref{pGeneratorsY2}.
\end{proposition}

\begin{proof}
After using Lemma \ref{lCompareMN} the same proof as that for Proposition \ref{pGeneratorsY2} applies mutatis mutandis.
\end{proof}

\subsection{The numerical Chow ring of $Y_4$}\label{ssY4}
Since $Y_4=Q\times Q$, this can be computed by the K\"unneth formula. $\Chownum^* (Y_4)$ is generated by $s, l, S, L$. 

\subsection{The numerical Chow rings of $Y_{ij}, Y_{ijk}$}\label{ssYijk}
The numerical Chow rings of double and triple intersections can be computed by the K\"unneth formula except for $\Chownum^* (Y_{23})$ which is $\Chownum^* (N)$. 

\subsection{Computing pushforwards and pullbacks via $\iota_{\{i,j,k\}, \{a,b\}} \colon Y_{ijk}\to Y_{ab}$}

These pushforwards and pullbacks are all easy to obtain using the K\"unneth formula except for $\iota_{\{1,2,3\},\{2,3\}}$ and $\iota_{\{2,3,4\},\{2,3\}}$.

\begin{lemma}\label{lPushPullSection}
Let $\EE$ be a vector bundle on a smooth projective variety $X$. Consider the diagram
\[
\xymatrix{
\P (\EE ) \ar[d]_{\pi}\\
X \ar@/_1pc/[u]_{\sigma}
}
\]
where $\sigma$ is a section and $\Sigma=\sigma (X)$. Then \[ \sigma_* (x) =\pi^* (x) \cdot \Sigma \] for $x\in \Chownum^* (X)$. For $y\in \Chownum^* (\P (\EE ))$ we have
\[
\sigma^* (y) =\pi_* (y\cdot \Sigma ).
\]
If $\EE$ is rank two and $y= \zeta \cdot c + d$ with $c,d\in \pi^* \Chownum^* (X)$ and $\zeta$ the relative hyperplane class, then 
\[
\pi_* (y) =c.
\]
\end{lemma}

\begin{proof}
All assertions follow directly from the fact that $\pi\mid_{\Sigma} \colon \Sigma \to X$ and $\sigma \colon X \to \Sigma$ are inverse isomorphisms.
\end{proof}

\begin{lemma}\label{lProjBundleClasses}
Let $\LL_1$ and $\LL_2$ be line bundles over a projective variety $X$ and let 
\[
\pi\colon \P=\P (\LL_1 \oplus \LL_2) \to X
\]
be the projective bundle of lines in the fibres of $\LL_1 \oplus \LL_2$. Denote by $\Sigma_1, \Sigma_2$ the sections of this bundle corresponding to $\P (\LL_1)$ and $\P (\LL_2)$, the lines contained in the fibres of $\LL_1$ and $\LL_2$, respectively. Then
\[
\OO_{\P}(\Sigma_i) =\OO_{\P}(1) \otimes \pi^* \LL_j
\]
for $\{i,j\}=\{1,2\}$.
\end{lemma}

\begin{proof}
The relative Euler sequence for the relative tangent bundle $\TT_{\P}$ 
\[
0 \to \OO_{\P} \to \pi^* (\LL_1\oplus \LL_2) \otimes \OO_{\P}(1) \to \TT_{\P} \to 0
\]
restricted to $\Sigma_i$ gives
\[
0\to \OO_{\Sigma_i} \to  (\LL_1\oplus \LL_2) \otimes \LL_i^{\vee} \to \NN_{\Sigma_i/\P}\to 0
\]
where $\NN_{\Sigma_i/\P}$ is the normal bundle of $\Sigma_i$ in $\P$; here we used $\OO_{\P}(1)\mid_{\Sigma_i} \simeq \LL_i^{\vee}$, which is correct since we are considering projective bundles of \emph{lines} in the fibres of vector bundles throughout. Thus
\[
\NN_{\Sigma_1/\P}\simeq \LL_2\otimes \LL_1^{\vee}, \quad \NN_{\Sigma_2/\P}\simeq \LL_1\otimes \LL_2^{\vee}. 
\]
On the other hand,
\[
\OO_{\P}(\Sigma_i) =\OO_{\P}(1) \otimes \pi^* \MM_i
\]
for some line bundles $\MM_i$ on the base $X$. Restricting both sides to $\Sigma_i$ we find
\[
\NN_{\Sigma_i/\P} = \LL_i^{\vee} \otimes \MM_i \simeq \LL_j \otimes \LL_i^{\vee}
\]
for $\{i,j\}=\{1,2\}$. It follows
\[
\OO_{\P}(\Sigma_i) =\OO_{\P}(1) \otimes \pi^* \LL_j
\]
as desired. 
\end{proof}

\begin{proposition}\label{pPushPull23}
We have
 \begin{align*}
 [\iota_{\{1,2,3\},\{2,3\}} (S\times S)] &=\xi +R_1+R_2 ,\\
 [\iota_{\{2,3,4\},\{2,3\}} (S\times S)] &=\xi -r_1-r_2  
\end{align*}
as classes in $\Chownum^* (N)$. 

Together with Lemma \ref{lPushPullSection}, this gives a complete description of pushforwards and pullbacks via $\iota_{\{1,2,3\},\{2,3\}}$ and $\iota_{\{2,3,4\},\{2,3\}}$.
\end{proposition}

\begin{proof}
The subvarieties $\iota_{\{1,2,3\},\{2,3\}} (S\times S)$ resp. $\iota_{\{2,3,4\},\{2,3\}} (S\times S)$  of $Y_{23}=N$ consist of all those normal directions to $S\times S$ in $L_C \times Q$ that are contained in $Y_{12}=L_C\times S$ resp. $Y_{24}=S \times Q$. Recall that $N$ is the projectivisation of 
\[
\NN_{(S\times S)/ (L_C\times Q)} \simeq \OO_{(\P^1\times\P^1)\times (\P^1\times\P^1)} (-1,-1, 0, 0) \oplus \OO_{(\P^1\times\P^1)\times (\P^1\times\P^1)} (0,0,1,1)
\]
and the (projectivisation of the) first summand here corresponds to normal directions contained in $Y_{12}=L_C\times S$ and the second to those contained in $Y_{24}=S \times Q$. 
The assertion follows from an application of Lemma \ref{lProjBundleClasses} because these are just the two natural sections $(1:0)$ and $(0:1)$ in this projective bundle.  \end{proof}

\subsection{Computing pushforwards and pullbacks via $\iota_{\{i,j\}, \{a\}} \colon Y_{ij}\to Y_{a}$}

The inclusions $\iota_{\{2,3\}, \{2\}} \colon Y_{23}\to Y_{2}$ and $\iota_{\{2,3\}, \{3\}} \colon Y_{23}\to Y_{3}$ are just the inclusions of the exceptional divisor into the blowup.

\begin{lemma}\label{lPushBlowUp}
Consider smooth varieties $Z \subset Y \subset X$ with $Z$ a divisor in $Y$ and $Y$ a divisor in $X$. Blowing up $Z$ in 
$X$ gives the following diagram:
\[
\xymatrix{
E \ar[d]^{\pi_E}  \ar[r]^{j_1} & \overline{Y} \cup E \ar[r]^{j_2}\ar[d]^{\pi_{\overline{Y} \cup E}} & \mathrm{Bl}_Z (X) \ar[d]^{\pi}\\
Z \ar[r]^{i_1} & Y \ar[r]^{i_2} & X
}
\]
We denote by $\pi_{\overline{Y}}  := \pi_{\overline{Y} \cup E}|_{\overline{Y}}$  the restriction of $\pi$ to the strict transform $\overline{Y}$ of $Y$. It is an isomorphism since $Z$ is a divisor in $Y$ and $Y$ is smooth. Let $j_{Y} : = j_2\mid_{\overline{Y}}\circ \pi_{\overline{Y}}^{-1}$.

Also we set $j := j_2 \circ j_1$.

Consider the map
\[
	(j_Y)_* \colon \Chow^*(Y) \to \Chow^*\bigl(\mathrm{Bl}_Z(X)\bigr). 
\]
According to the above definitions we have  $(j_Y)_* := (j_2\mid_{\overline{Y}})_* \circ \pi_{\overline{Y}}^*$. Then for $y \in \Chow^*(Y)$ we have
\[
	(j_Y)_*(y) = \pi^* (i_2)_* (y) - j_*\pi_E^*(i_1)^*(y).
\]
\end{lemma}

\begin{proof}
By the Moving Lemma \cite[Appendix A, Lemma A.1 (a)]{E-H16}, a class $y \in \Chow^*(Y)$ can be represented by a cycle $Z_y$ that intersects the codimension one subvariety $Z$ of $Y$ generically transversely.  By additivity it suffices to prove the assertion of the Lemma for one irreducible reduced component of the support of $Z_y$. We call this $\Gamma$. Then the scheme-theoretic pullback $\pi^{-1}(\Gamma )$ decomposes as 
\[
\pi^{-1}(\Gamma ) = \Gamma_{\overline{Y}} \cup \Gamma_{E}
\]
where $\Gamma_{\overline{Y}}$ is the preimage of $\Gamma$ on $\overline{Y}$ under the isomorphism $\overline{Y}\to Y$, and $\Gamma_{E}$ are the remaining components, all of which have support on $E$. Now
\[
[ \pi^{-1}(\Gamma ) ] = \pi^* (i_2)_* ([\Gamma])
\]
by \cite[Thm. 1.23, (a)]{E-H16} (the hypotheses of this Theorem are satisfied since $\Gamma$ is generically transverse to $Z$ because we applied the Moving Lemma). 
We have
\[
[\Gamma_{\overline{Y}}] = (j_Y)_*([\Gamma])
\]
and
\[
[\Gamma_{E}] = j_*\pi_E^*(i_1)^*([\Gamma])
\]
because by our choice $\Gamma \cap Z$ has an underlying cycle representing $[\Gamma] . Z = (i_1)^*([\Gamma])$. 
\end{proof}

With this Lemma we can compute pushforwards and pullbacks for:
\begin{enumerate}
\item
$\iota_{\{1,2\}, \{2\}} \colon Y_{12}\to Y_{2}$.
\item
$\iota_{\{1,3\}, \{3\}} \colon Y_{13}\to Y_{3}$.
\item
$\iota_{\{2,4\}, \{2\}} \colon Y_{24}\to Y_{2}$.
\item
$\iota_{\{3,4\}, \{3\}} \colon Y_{34}\to Y_{3}$.
\end{enumerate}

\medskip

The inclusions $\iota_{\{2,4\}, \{4\}} \colon Y_{24}\to Y_{4}$ and $\iota_{\{3,4\}, \{4\}} \colon Y_{34}\to Y_{4}$ can be handled using the K\"unneth formula. 

\medskip

Lastly, we have to consider $\iota_{\{1,2\}, \{1\}} \colon Y_{12}\to Y_{1}$ and $\iota_{\{1,3\}, \{1\}} \colon Y_{13}\to Y_{1}$. Here everything follows using the K\"unneth formula except the pullback of $D$:

\begin{lemma}\label{lPullbackD}
We have 
\begin{align*} 
\iota_{\{1,2\}, \{1\}} ^* (D) &= e R_1R_2+3f(R_1+R_2),  \\
\iota_{\{1,3\}, \{1\}} ^* (D) &= r_1r_2 E+3 (r_1+r_2) F .
\end{align*}
\end{lemma}

\begin{proof}
We only need to prove one of these formulas because the other follows by symmetry. We have the two inclusion
\[
j:=\iota_{\{1,2,3\}, \{1,2\}} \colon S\times S \to L_C\times S, \quad i:=\iota_{\{1,2\}, \{1\} }\colon L_C\times S \to L_C\times L_C .
\]
Now the pullback $(i\circ j)^* (D)$ is the diagonal $\Delta_C \subset C\times C\subset S\times S$. Hence, intersecting with a basis of the divisors in $S\times S$, we find
\[
(i\circ j)^* (D) = 3(r_1 +r_2)R_1R_2 + 3 r_1r_2 (R_1+R_2) . 
\]
Therefore, $i^* (D)$ must be of the form
\[
\alpha h R_1R_2 + \beta e R_1 R_2 + \gamma h^2 (R_1+R_2) + \delta f (R_1 +R_2)
\]
with $\alpha, \beta, \gamma,\delta$ integers. Now we have the equation
\[
D \cdot i_* (L_C\times S) = i_* (i^* (D)).
\]
Using Proposition \ref{pLCxLCdsg} and the fact that we can compute $i_*$ using the K\"unneth formula, we calculate the images under $i_*$ of $hR_1R_2$, $e R_1 R_2$, $h^2 (R_1+R_2)$ and $f (R_1 +R_2)$ and find that the images are linearly independent. We obtain linear equations for the unknowns $\alpha, \beta, \gamma , \delta$. The only solution is the one in the statement. This is checked in \cite[\textsf{pullbackOfD.m2}]{BBG-M2}.
\end{proof}

\begin{theorem}\label{tCubicThreefoldDegeneration}
We have
\[
	\Num_{\mathrm{prelog}}^3(Y) = \Z^6
\]
modulo torsion.
\end{theorem}

\begin{proof}
For this computation we use Definition \ref{dBothmer}: 
 \xycenter{
  & 0 \ar[d] & 0 \ar[d] & \\
  & R(Y) \ar[d] \ar[r] & \Num^3_{\mathrm{prelog}}(Y) \ar[d] \ar[r] & 0 \\
\bigoplus \Num^2(Y_{ij}) \ar[d]^-{\rho'} \ar[r]^\delta
&\bigoplus \Num^3(Y_{i})  \ar[r] \ar[d]^-\rho
&\coker \delta \ar[r]
& 0  \\
\bigoplus \Num^2(Y_{ijk}) \ar[r]^-{\delta'}
&\bigoplus \Num^3(Y_{ij})
 }
with explicitly given maps $\rho, \rho'$ in terms of the push forwards $\iota_*$ and $\delta,\delta'$ in terms of the pullbacks $\iota^*$ calculated above. 

For the convenience of the reader we give a slow walk through the necessary computations.
A Macaulay2 script doing the same work is available at \cite[\textsf{prelogY.m2} using \textsf{prelogYparts.m2}]{BBG-M2}.

Using Propositions \ref{pLCxLCdsg}, \ref{pGeneratorsY2}, \ref{pGeneratorsY3} and Subsections \ref{ssY4}, \ref{ssYijk} we can calculate the intersection rings of $Y_{i}, Y_{ij}$ and $Y_{ijk}$ in degree $3$:

\xycenter{
  & 0 \ar[d] & 0 \ar[d] & \\
  & R(Y) \ar[d] \ar[r] & \Num^3_{\mathrm{prelog}}(Y) \ar[d] \ar[r] & 0 \\
\bigoplus \Z^{32} \ar[d]^-{\rho'} \ar[r]^\delta
&\bigoplus \Z^{39}  \ar[r] \ar[d]^-\rho
&\coker \delta \ar[r]
& 0 \\
\bigoplus \Z^{12} \ar[r]^-{\delta'}
&\bigoplus \Z^{32}
 }

A good check that we got everything right is that indeed $\delta\rho = \rho' \delta'$ (this is
the Friedman condition).

Now one can check that $\delta$ has rank $22$ in every characteristic except $2$ where
it has rank $21$.  The same is true for $\rho$.

This reduces the diagram to:
\xycenter{
  & 0 \ar[d] & 0 \ar[d] & \\
  & \Z^{17} \ar[d]^{\sigma} \ar[r] & \Num^3_{\mathrm{prelog}}(Y) \ar[d] \ar[r] & 0 \\
\bigoplus \Z^{32} \ar[d]^-{\rho'} \ar[r]^\delta
&\bigoplus \Z^{39}  \ar[r]^-{\gamma} \ar[d]^-\rho
&\Z^{17}  \oplus \Z/2 \ar[r] 
& 0 \\
\bigoplus \Z^{12} \ar[r]^-{\delta'}
&\bigoplus \Z^{32}
 }
The degree $3$ part of the numerical prelog ring $\Num_{\mathrm{prelog}}^3(Y)$, modulo torsion, is therefore the image of an
explicitly given $17 \times 17$ matrix $M$. We calculate that this matrix $M$ has rank $6$. Therefore
\[
	\Num_{\mathrm{prelog}}^3(Y) = \Z^6
\]
modulo torsion.
\end{proof}

We now identify explicit effective generators of $\Num_{\mathrm{prelog}}^3(Y)$ modulo torsion.

\begin{theorem}\label{tGeneratorsY}
The following vectors in $\bigoplus_{i=1}^4\Chownum^3 (Y_i)$ satisfy the prelog condition and are mapped to a $\Z$-basis of $\Num_{\mathrm{prelog}}^3(Y)$, modulo torsion:
\begin{align*}
Z_{03} = (&h^3,h^3,0,0) \\
Z_{30} = (&H^3,0,H^3,0) \\
Z_{12} = \bigl(&(h^2-2f)H,(h^2-2f)S,0,0\bigr) \\
Z_{21} = \bigl(&h(H^2-2F),0,s(H^2-2F),0\bigr) \\
Z_{\Delta} = \bigl(
	&h^3 + h^2H + hH^2 + H^3 - D,\\
	&N_{r_1r_2}+N_{R_1r_2}+N_{r_1R_2}+N_{R_1R_2}, \\
	&M_{r_1r_2}+M_{R_1r_2}+M_{r_1R_2}+M_{R_1R_2},\\
	&sl+sL+Sl+SL
	\bigr) \\
Z_D = (& D-eF-fE, 0,0,0) \\
\end{align*}
 \end{theorem}	

\begin{proof}
In \cite[\textsf{prelogY.m2} using \textsf{prelogYparts.m2}]{BBG-M2} we check that the elements in the statement satisfy the prelog condition and are a $\Z$-basis: we do this 
by showing that they are linearly independent and writing the images under $\gamma\circ\sigma$ of all standard generators of $R(Y)\simeq \Z^{17}$ as $\Z$-linear combinations of the elements above (modulo torsion).
\end{proof}

\begin{proposition}\label{pIntFormOnNumPrelog}
The intersection form in $\Num_{\mathrm{prelog}}^3(Y)$ with respect to this basis is
\[
\begin{pmatrix}
 0 &1& 0 &0& 1 & 0  \\
 1& 0 &0& 0& 1&  0  \\
 0 &0 &0 &1 &1 & 0  \\
 0 &0 &1 &0 &1  &0  \\
 1 &1 &1 &1 &-6 &-8 \\
 0 &0 &0 &0 &-8 &-8 \\
\end{pmatrix}.
\]
\end{proposition}

\begin{proof}
By the definition of the ring structure on $\Num_{\mathrm{prelog}}^*(Y)$ given in \cite{BBvG19-1}, the intersection of two prelog cycles can be computed componentwise. We have $\Num_{\mathrm{prelog}}^6(Y) \simeq \Z$ since $Y$ is connected.  
The intersection matrix is then calculated in \cite[\textsf{prelogY.m2} using \textsf{prelogYparts.m2}]{BBG-M2}.
\end{proof}

\begin{remark}\label{rVoisin}
Notice that for a very general cubic threefold $V$ it is known that the Chow group of $V\times V$ modulo numerical equivalence is generated by (point)$ \times V$, $V\times$(point), (line)$\times$(hyperplane section), (hyperplane section)$\times$(line), and the diagonal. We would like to thank Claire Voisin for pointing this out.  Our specialisation homomorphism to the prelog Chow group is therefore injective, but not surjective, in this case; this is made more precise in the next proposition, Proposition \ref{pSpecialisationCyclesOnVxV}. 
\end{remark}

\begin{proposition}\label{pSpecialisationCyclesOnVxV}
The images of $Z_{03}$, $Z_{30}$, $Z_{12}$, $Z_{21}$, and $Z_{\Delta}$ in $\Num_{\mathrm{prelog}}^3(Y)$ are the specialisations of $(point) \times V$, $V\times (point)$, (line)$\times$(hyperplane section), (hyperplane section)$\times$(line), and the diagonal, respectively.
\end{proposition}

\begin{proof}
The specialisation map is given as follows: the cycles in the statement of the Proposition give relative cycles in the family with central fibre removed; to get the specialisation for each of those cycles, we take its closure and intersect it with the central fibre.

Consider the cycle $(point) \times V$. Specialising the point to a point in $L\simeq \P^3$ away from $S$, gives the image of $Z_{03}$; for the numbering of the components it is helpful to look back at Figure \ref{ProductConfiguration}. 
For $Z_{30}$, $Z_{12}$, $Z_{21}$ the argument is then similar. That $Z_{\Delta}$ is the specialisation of the diagonal can be seen as follows. The class
\[
sl+sL+Sl+SL
\]
is the class of the diagonal on $Y_4=Q\times Q$ whereas 
\[
h^3 + h^2H + hH^2 + H^3 - D
\]
is the class of the diagonal on $L_C\times L_C$: indeed, $h^3 + h^2H + hH^2 + H^3$ is the class of $\Delta_{\P^3}$ on $\P^3\times \P^3$, and $L_C\times L_C$ is obtained from $\P^3\times \P^3$ by first blowing up $C\times \P^3$ and then $L_C\times C$. The diagonal intersects $C\times \P^3$ in $\Delta_C\subset C\times C$, and hence as schemes
\[
\sigma^{-1} (\Delta_{\P^3}) = \Delta_{L_C} \cup D
\]
where $\sigma \colon L_C\times L_C\to \P^3\times \P^3$ is the composition of the two blowups. Since $D$ is three-dimensional, we have $\sigma^* ([\Delta_{\P^3}]) = [\Delta_{L_C}]+  D$ as cycle classes. To justify the components of $Z_{\Delta}$ on $Y_2$ and $Y_3$, note that the diagonal in the product family $\XX\times_B \XX \to B$ intersects  $S\times S\subset X\times X$ in $\Delta_S$. Note also that $S\times S$ is precisely the locus in which the total space $\XX\times_B \XX$ is singular. The class of $\Delta_S$ in $S\times S$ is
\[
r_1r_2+R_1r_2+ r_1R_2 + R_1R_2
\]
and pulling this back to $Y_{23}$ and pushing forward to $Y_2$ and $Y_3$ we obtain the middle two entries in $Z_{\Delta}$.
\end{proof}

We now calculate the saturated numerical prelog Chow group in degree $3$. 

\begin{theorem}\label{tSaturatedPrelogY}
$\Num_{\mathrm{prelog, sat}}^3 (Y)\simeq \Z^6$ and is generated by the classes in Theorem \ref{tGeneratorsY} and a half of their sum. 
\end{theorem}

\begin{proof}
The group $\Num_{\mathrm{prelog, sat}}^3 (Y)$ is the saturation of $\Num_{\mathrm{prelog}}^3 (Y)$ in the lattice $\mathrm{coker}\,\delta / (torsion)\simeq \Z^{17}$. The calculation is done in \cite[\textsf{prelogY.m2} using \textsf{prelogYparts.m2}]{BBG-M2}: writing the generators given in Theorem \ref{tGeneratorsY} in the standard basis of $\Z^{17}$ gives a $17\times 6$ matrix $N$. The gcd of its maximal minors is equal to $2$. Therefore, $N$ has full rank in every characteristic except $2$. Moreover, in characteristic $2$, $N$ has rank $5$. The kernel of $N$ in characteristic $2$ is generated by the sum of the $6$ generators of Theorem \ref{tGeneratorsY}. 
\end{proof}

\providecommand{\bysame}{\leavevmode\hbox to3em{\hrulefill}\thinspace}
\providecommand{\href}[2]{#2}

\end{document}